%% file: MIChaZero.tex
\documentclass[10pt,a4paper,reqno]{amsart} \input{preamble}

\begin{document}

\title{A definitely periodic chain in the integral Lie ring of partitions} 
 \author[R.~Aragona]{Riccardo Aragona}
\author[R.~Civino]{Roberto Civino}
\author[N.~Gavioli]{Norberto Gavioli}

\address{DISIM \\
 Universit\`a degli Studi dell'Aquila\\
 via Vetoio\\
 I-67100 Coppito (AQ)\\
 Italy}       

\email[R.~Aragona]{riccardo.aragona@univaq.it}
\email[R.~Civino]{roberto.civino@univaq.it} 
\email[N.~Gavioli]{norberto.gavioli@univaq.it}

\date{} \thanks{All the authors are members of INdAM-GNSAGA
 (Italy). R. Civino is funded by the Centre of excellence
 ExEMERGE at University of L'Aquila, with which also the other authors collaborate.\\
   \emph{Data availability:} all data generated or analyzed during this study are included in this published article.}


\begin{abstract}
Given an integer $n$, we introduce the integral Lie ring of partitions with bounded maximal part, whose 
elements are in one-to-one correspondence to integer partitions with parts in $\{1,2,\dots, n-1\}$. 
Starting from an abelian subring, we recursively define a chain of idealizers and we prove that the sequence of
ranks of consecutive terms in the chain is definitely periodic. Moreover, we show that its growth depends of the partial sum of the partial sum of the sequence
counting the number of partitions.
This work generalizes our previous recent work on the same topic, devoted to the modular case where 
partitions were allowed to have a bounded number of repetitions of parts in a ring of coefficients of positive characteristic.
\end{abstract}
\begin{nouppercase}
\maketitle
\end{nouppercase}

\section{Introduction}
Given an integer $n\geq 3$, we recently defined a Lie ring structure on the set of partitions with parts in $\{1,2,\dots,n-1\}$ and where each part is allowed to have at most $m-1$ repetitions, for some given $m > 2$. In the obtained structure, here called $\Lie_m(n)$, we  recursively defined a chain of idealizers $(\mathfrak N_i)_{i \geq -1}$ starting from an abelian subring $\mathcal T$. 
We proved that the rank of \(\mathfrak{N}_{i}/\mathfrak{N}_{i-1}\) as free $\Z_m$-module is $q_{i+1}$, where $q_i$ is the partial sum of sequence counting the number of partitions of $i$ into at least two parts, each allowed to be repeated at most $m-1$ times~\cite{Aragona2023}. 
This was done to address the problem,  proposed by Aragona et al.\ in 2021~\cite{aragona2021rigid} and left unsolved, of computing the growth of a chain of normalizers in a Sylow $m$-subgroup starting from an elementary abelian regular subgroup $T <\Sym(m^n)$, when $m$ is an odd prime. Indeed, 
 $\Lie_m(n)$ was constructed as the iterated wreath product of Lie rings of rank one~\cite{Aragona2023} and, when $m$ is prime, this corresponds exactly to the construction of the graded Lie algebra associated to the lower central series of a Sylow $m$-subgroup of $\Sym(m^n)$~\cite{MR2148825}. In this construction, the abelian subring $\mathfrak{N}_{-1} = \mathfrak T$ at the base of the idealizer chain corresponds to the elementary abelian regular subgroup $T$ at the base of normalizer chain. It is important to stress that the combinatorial equality mentioned above is valid only for the first $n-2$ terms of the chain, and the problem of understanding the general behavior of the chain is out of reach at the time of writing. 

In this work we address a similar problem in the case of characteristic zero, i.e.\ in a Lie ring $\Lie(n)$ with integer coefficients of partitions with parts in $\{1,2,	\dots, n-1\}$. The main combinatorial difference between the two settings is that now each part is allowed to have an unbounded number of repetitions. We show that this significantly affects how  the idealizer chain grows. In particular, we prove here that the sequence of consecutive quotient ranks in the idealizer chain depends on the \emph{second partial sum} of the sequence of the integer partitions and is (definitely) periodic.

We conclude this section by introducing the construction of the partition Lie ring with integer coefficients and by showing some preliminar properties.

\subsection{Notation and preliminaries}
The construction of the integral ring of partitions presented below is based on the construction of the 
\emph{Lie ring over $\Z_m$ of partitions with bounded maximal part}~\cite{Aragona2023}, here customized for \(m=0\).
We advise the reader to also refer to the original paper for further details and to the papers for the original problem
in the setting of the symmetric group~\cite{Aragona2019,aragona2021rigid}.
		
Let  \(\Lambda   =  \{\lambda_i\}_{i=1}^\infty\)  be  a   sequence  of
non-negative integers with finite support, i.e.\ such that
\[\wt(\Lambda)= \sum_{i=1}^\infty i\lambda_i < \infty.\] The sequence
\(\Lambda\) defines  a \emph{partition} of \(N  = \wt(\Lambda)\). Each
non-zero  $i$  is   a  \emph{part}  of  the   partition,  and the  integer
\(\lambda_i\) is the multiplicity of the part \(i\) in \(\Lambda\).  The maximal  part of
\(\Lambda\) is the  number \(\max\bigl(\{i  \mid \lambda_i  \ne 0\}\bigr)\). The set
of the  partitions whose maximal part  less than or equal to \(j\)  is denoted by
\(\Part(j)\).

The  \emph{power  monomial}  \(x^\Lambda\),  where  \(\Lambda\)  is  a
partition, is defined as \(x^\Lambda= \prod_{i} x_i^{\lambda_i}\).  
Given a positive integer \(n\) and denoting by  \(\partial_k\)  the derivation given by the standard  partial
derivative with respect to $x_k$, where $1 \leq k \leq n$, we
define by \(\Lie(n)\) the free \(\Z\)-module spanned by the basis
\[\mathcal{B}:=\Set{x^\Lambda\partial_k \mid 1\le k \le  n \text{ and }
    \Lambda\in            \Part(k-1)},\]           and  we            set
\(\mathcal{B}_u:= \Set{x^\Lambda\partial_k \in \mathcal{B} \mid k=u}\).
The module \(\Lie(n)\)  is endowed with a structure of  Lie ring, where the
Lie bracket is defined on the basis \(\mathcal{B}\) by
\begin{eqnarray*}
  \left[x^\Lambda \partial_k , x^\Theta \partial_j\right] :=&
                                                             \partial_{j}(x^\Lambda)  x^\Theta \partial_k - x^\Lambda \partial_{k}(x^\Theta) \partial_j \nonumber \\
  =&
     \begin{cases}
       \partial_{j}(x^\Lambda)  x^\Theta \partial_k & \text{if \(j<k\)}, \\
       - x^\Lambda \partial_{k}(x^\Theta) \partial_j & \text{if \(j>k\)},\\
       0 & \text{otherwise.}
     \end{cases}
\end{eqnarray*} 
 This operation  is then extended to a Lie  product on  \(\Lie(n)\)  by bilinearity, and the resulting structure \(\Lie(n)\) is called the
  \emph{integral Lie ring of partitions with parts in $\{1,2,\dots, n-1\}$}.
 
In the remainder of the paper we deal with homogeneous  subrings of  \(\Lie(n)\), which are defined as follows.
\begin{definition}
 A Lie  subring \(\mathfrak{H}\) of 
\(\Lie(n)\) is said to be \emph{homogeneous}  if it is the free \(\Z\)-module
spanned by some subset \(\mathcal{H}\) of \(\mathcal{B}\). 
\end{definition}
The following result on homogeneous subrings can be proved as in the case of the modular ring $\Lie_m(n)$~\cite[Theorem 2.5]{Aragona2023}.
Here 
and the remainder of the paper, if \(\mathcal{H}\) is a subset  of \(\mathcal{B}\), then its \emph{idealizer}
is defined as
\[  N_\mathcal{B}(\mathcal{H}):=  \Set{b\in  \mathcal{B} \mid  [b,h] \in
    \Z\mathcal{H} \text{ for all } h\in \mathcal{H}  }.
\]

\begin{theorem}\label{thm_homogeneous:normalizers}
  Let  \(\mathfrak{H}\) be  a homogeneous  subring of  \(\Lie\) having
  basis  \(\mathcal{H}  \subseteq  \mathcal{B}\).   The  idealizer  of
  \(\mathfrak{H}\)  in  \(\Lie(n)\)  is  the  homogeneous  subring  of
  \(\Lie(n)\)  spanned  by  \(N_\mathcal{B}(\mathcal{H})\) as  a  free
  \(\Z\)-module.
\end{theorem}

The chain of idealizers in \(\Lie_m(n)\) started from the abelian homogeneous Lie  subring \(\mathfrak{T}=\Span{\partial_1,\dots,\partial_n}\)~\cite{Aragona2023}. Here we deal with the idealizer chain defined in $\Lie(n)$ starting from $\mathfrak{T}$ and defined as follows:
\begin{equation}\label{def_chain}
\mathfrak{N}_i = 
\begin{cases}
\mathfrak{T} & i=-1,\\
N_{\Lie(n)}(\mathfrak{T}) & i=0,\\
N_{\Lie(n)}(\mathfrak{N}_{i-1}) & i \geq 1.
\end{cases}
\end{equation}

\subsection{Organization of the paper}
The remainder of the paper is organized as follows:
Sec.~\ref{sec:comb} is devoted to combinatoric aspects of $\Lie(n)$: we introduce a chain of subsets $(\mathcal N_i)_{i \ge -1}$ and we show 
that the cardinalities of $\mathcal{N}_i\setminus\mathcal{N}_{i-1}$ depend, \emph{up to periodicity}, on the second partial sum of the sequence of integer partitions (cf.\ Corollary~\ref{cor:lev_sizes} and Corollary~\ref{cor:main}).
In Sec.~\ref{sec:norm}, we show (cf.\ Theorem~\ref{thm:main}) that  the free \(\Z\)-modules spanned by
	the sets \(\mathcal{N}_i\)s coincide with the idealizers of Eq.~\eqref{def_chain}, yielding our main contribution of Corollary~\ref{coro:main2}, which connects, up to periodicity, the rank of the   free    \(\Z\)-modules
  \(\mathfrak{N}_i/\mathfrak{N}_{i-1}\) with the second partial sum of the sequence of integer partitions.
  

\section{Combinatorics of the integral Lie ring of partitions}\label{sec:comb}
The techniques developed for the case of positive characteristic do not fit well in the case of the integral Lie ring of partitions, which requires new combinatorial tools.
As we will show in the remainder of this paper, the combinatorial properties of \(\Lie(n)\) depend on the behavior of the functions introduced below.
\begin{definition}\label{def:indices}
Let \(i\ge -1\) be an integer and   let  \(1\le r_i  \le n-1\)  be such  that
\(i\equiv r_i \bmod (n-1)\). Write
\begin{equation*}
  i=(h_i-1)(n-1)+r_i,
\end{equation*}
where        \begin{equation*}
h_i:=\left\lfloor
    \dfrac{i-1}{n-1}\right\rfloor+1.
\end{equation*} 
The         \emph{weight-degree         function} is defined as
\(\WD\colon \Z\mathcal{B} \to \Z \) by
\begin{equation*}
  \WD(c \cdot x^\Lambda\partial_{k}) := \wt(\Lambda)-\deg(x^\Lambda)+n-k,
\end{equation*} 
and 
 the  \emph{\(i\)-th  level  function}
\(\lev_i\colon \Z\mathcal{B} \to \Z \) by
\begin{equation*}
  \lev_i(c \cdot x^\Lambda\partial_{k}) := h_i\WD(\Lambda)+\deg(x^\Lambda)-1.
\end{equation*}
\end{definition}
We  now show  that the weight-degree  function   is
  bounded.
\begin{lemma}   \label{lem:bounds}  
Let $x^\Lambda\partial_{k} \in \mathcal{B}$. Then 
 \(\WD(x^\Lambda\partial_{k}) \geq n-k\) 
  and 
  \(\WD(x^\Lambda\partial_{k})      =     n-k\) if and only if \(x^\Lambda\partial_{k}=x_1^{\lambda_1}\partial_k\).
   Moreover, assume that
  \(\lev_i(x^\Lambda\partial_{k})\le   i\)   for  some   \(i\).  Then
  \(\WD(x^\Lambda\partial_{k})         \leq         n-1\)         and
  \(\WD(x^\Lambda\partial_{k}) = n-1\) if and only if
  \(x^\Lambda\partial_k=\partial_1\).
\end{lemma}

\begin{proof}
  Note                                                            that
  \(\WD(x^\Lambda\partial_k)=
  \left(\wt(\Lambda)-\deg(x^\Lambda)\right)  +  n-k   >  n-k\)  unless
  \(\wt(\Lambda)=\deg(x^\Lambda)\),    which    is    equivalent    to
  \(x^\Lambda\partial_{k}=x_1^{\lambda_1}\partial_k\).
	
 Let now  $i$ be such  that \(\lev_i(x^\Lambda\partial_{k})\le   i\) and assume   \(\WD(x^\Lambda\partial_{k})   \ge    n-1\). Then
  \begin{align*}
  i &\ge                \lev_i(x^\Lambda\partial_{k})       \\         
  &= h_i\WD(x^\Lambda\partial_{k})+\deg(x^\Lambda)-1            \\
        &\ge h_i(n-1)+\deg(x^\Lambda)-1                                         \\
  &= (n-1)+(h_i-1)(n-1)+r_i+\deg(x^\Lambda)-r_i-1\\
  &=i+(n-1)-r_i+\deg(x^\Lambda)-1.
  \end{align*}    This     implies     that
  \((n-1)-r_i+\deg(x^\Lambda)-1 \le 0\).  Since  \(r_i\le  n-1\),  this  is
  possible  only if  either \(r_i=n-1\)  and \(\deg(x^\Lambda) \le 1\)  or
  \(r_i=n-2\)   and   \(\deg(x^\Lambda)=0\).    
  
  In  the   first   case we have either \(k=1\) and hence \(x^\Lambda\partial_k=\partial_1\), as required, or 
  \(x^\Lambda=x_j^{\lambda_j}\),      where       \(1\le j<k\) and \(\lambda_j\le 1\).       This      implies
  \(n-1\le   \WD(x^\Lambda\partial_k)=   (j-1)\lambda_j+n-k \le j+n-k-1\)  and   consequently
  \(k \le j\),  which  contradicts \(j<k\).  
  
   In  the  second  case we have
  \(x^\Lambda\partial_k=\partial_k\)                               and
  \(n-1 \le  \WD(x^\Lambda\partial_k)=n-k\). This leads again to  \(k=1\) and
  \(x^\Lambda\partial_k=\partial_1\).
\end{proof}

\begin{proposition}\label{prop:lev_wd_equations}
  If  \(x^\Lambda\partial_k, x^\Theta\partial_u \in \mathcal B\)   do  not
  commute, then
  \begin{equation}\label{eq:horizontal_operation}
    \begin{split}
      \lev_i([x^\Lambda\partial_{k},          x^\Theta\partial_{u}])&=
      \lev_i(x^\Lambda\partial_{k})                                  +
      \lev_i(x^\Theta\partial_{u})-h_i(n-1)\\ 
    \end{split}
  \end{equation}
  and
  \begin{equation}\label{eq:wd_decreasing}
    \begin{split}
      \WD([x^\Lambda\partial_{k},             x^\Theta\partial_{u}])&=
      \WD(x^\Lambda\partial_{k}) + \WD(x^\Theta\partial_{u}) -(n-1).
    \end{split}
  \end{equation}	
  Moreover,     if     \(\lev_i(x^\Lambda\partial_k)\le    i\)     and
  \(\lev_j(x^\Theta\partial_u)\le j\) for some \(i,j\ge -1\), then
  \begin{equation}\label{eq:wd_decreasing}
    \begin{split}
      \WD([x^\Lambda\partial_{k},      x^\Theta\partial_{u}])      \le
      \min(\WD(x^\Lambda\partial_{k}) , \WD(x^\Theta\partial_{u})),
    \end{split}
  \end{equation}
  with    equality       if    and   only    if   one    of
  \(x^\Lambda\partial_{k}\)     or    \(x^\Theta\partial_{u}\)     is
  \(\partial_1\).
\end{proposition}
\begin{proof}
  Let 
  \(x^\Gamma\partial_u=[x^\Lambda\partial_{k}, x^\Theta\partial_{u}]\). We have 
  \begin{align*}
  	\WD(x^\Gamma\partial_u)&=\wt(\Lambda)+\wt(\Theta)-k -\deg(x^\Lambda)- \deg(x^\Theta) + 1 +n-u\\
  	&=\WD(x^\Lambda\partial_k)+\WD(x^\Theta\partial_u) -(n -1).
  \end{align*} 
	As a consequence
  \begin{align*}
  	\lev_i(x^\Gamma\partial_u)=& h_i\WD(x^\Gamma\partial_u) + \deg (x^\Gamma)-1 \\
  	=& h_i\bigl(\WD(x^\Lambda\partial_k)+\WD(x^\Theta\partial_u) -(n -1)\bigr)  
  	+ \deg (x^\Lambda) +\deg (x^\Theta)-2 \\
  	=&           \lev_i(x^\Lambda\partial_u)+\lev_i(x^\Theta\partial_k)
  	-(n-1)h_i.
  \end{align*}
  The last part  of the claim is a straightforward  consequence of Lemma~\ref{lem:bounds}.
\end{proof}

\begin{remark}\label{rem:increase}
  A direct check shows that
  \begin{equation*}
    \lev_j(x^\Lambda\partial_{k}) = 
    \lev_{i}(x^\Lambda\partial_{k}) +(h_j-h_i)\WD(x^\Lambda\partial_k).
  \end{equation*}
\end{remark}

\begin{lemma}\label{lem:lev_reduction}
  If  \(\lev_i(x^\Lambda\partial_k)   \le  i  \),  then   there  exists
  \(j\le i\) such that \(\lev_j(x^\Lambda\partial_k)=j\).
\end{lemma}

\begin{proof}
  We argue by induction on \(i\). If \(i=-1\), then \(h_{-1}=0\) and hence
  \[\lev_{-1}(x^\Lambda\partial_{k}) =  \deg(x^\Lambda)-1 \le  -1\] so
  \(\deg(x^\Lambda)=0\) and \(\lev_{-1}(x^\Lambda\partial_{k}) = -1\).

  Let now      \(i\ge      0\).        We      can       assume      that
  \(\ell:=\lev_i(x^\Lambda\partial_k)     <     i\).     Note     that
  \[\lev_{\ell}(x^\Lambda\partial_{k})=
  \lev_{i}(x^\Lambda\partial_{k})+
  (h_\ell-h_i)\WD(x^\Lambda\partial_{k})                           \le
  \lev_{i}(x^\Lambda\partial_{k})=\ell.\]                               If
  \(\lev_{\ell}(x^\Lambda\partial_{k})=\ell\), then we  have proved the claim,
  otherwise
  \(\lev_{\ell}(x^\Lambda\partial_{k})                               <
  \lev_{i}(x^\Lambda\partial_{k})=\ell   <  i\)   and  the   induction
  hypothesis          is           satisfied. Therefore
  \(\lev_{j}(x^\Lambda\partial_{k}) = j\) for some \(j\le \ell <i\).
\end{proof}

\begin{lemma}\label{lem:lev_plus_n-1}
  If     \(\lev_i(x^\Lambda\partial_k)      \le     i      \),     then
  \(\lev_{i+(n-1)}(x^\Lambda\partial_k) \le i+(n-1)  \). In particular,
  \(\lev_i(x^\Lambda\partial_k)  >   i\)  for  all  
  \(i\le h(n-1)\)  if and only if  \(\lev_i(x^\Lambda\partial_k) > i\)
  for all \(i\) such that \((h-1)(n-1)+1 \le i\le h(n-1)\).
\end{lemma}
\begin{proof}
  By    Lemma
  \ref{lem:bounds}   we have  that    \(\WD(x^\Lambda\partial_k)\le   n-1\) .        Thus      \(h_{i+(n-1)}=h_{i}+1\)       and
  \[\lev_{i+(n-1)}(x^\Lambda\partial_k)  = \WD(x^\Lambda\partial_k)  +
  \lev_i(x^\Lambda\partial_k) \le (n-1)+i.\]
\end{proof}
\subsection{Introducing the chain}
We are now ready to introduce a chain of subsets of \(\mathcal{B}\). We will prove later that they span as free modules the idealizers of Eq~\eqref{def_chain}.
\begin{definition}\label{def:setchain}
  For \(i \ge -1 \)
  let 
  \begin{equation*}
    \mathcal{N}_i:=\Set{ x^\Lambda\partial_{k} \mid
      \lev_j(x^\Lambda\partial_k) \le  j \text{ for some } j\le i }.
  \end{equation*}
  Moreover, for $i \ge 0$ let 
 \[ \mathcal{L}_i:=\mathcal{N}_i\setminus\mathcal{N}_{i-1}.\]
\end{definition}

\begin{remark}\label{rem:lev_equlity}
 Note            that     $\mathcal{N}_{-1}=\Set{\partial_1,\dots,\partial_n}$, that $\mathfrak{T}$ is the free $\Z$-module spanned by   $\mathcal{N}_{-1}$, 
 and that the subsets  \(\Set{\mathcal{N}_i}_{i\ge -1}\) constitute  an ascending chain of \(\mathcal{B}\).
  Note also that an easy application of Lemma~\ref{lem:lev_reduction} gives
  \begin{equation*}
    \mathcal{N}_i=\Set{ x^\Lambda\partial_{k} \mid
      \lev_j(x^\Lambda\partial_k) =  j \text{ for some } j\le i }.
  \end{equation*}
Furthermore, \(x^{\Lambda}\partial_k\in \mathcal{L}_i\) if and only if \(i\) is minimum such that \(\lev_i(x^\Lambda\partial_k) =  i \), i.e.\
  \[
    \mathcal{L}_i  =   \Set{
      \vphantom{\Big\{}  x^{\Lambda}\partial_{k} \in  \mathcal{B} \: \middle\vert\:
      i= \min_j \Set{j=\lev_j(x^{\Lambda}\partial_{k})} }.
  \]
\end{remark}

\begin{proposition}\label{prop:rem_bound}
  Let \(i\ge 0\).   If \(x^\Lambda\partial_{k}\in \mathcal{L}_i\), then
  \(r_i   >   \WD(x^\Lambda\partial_{k})\).    In   particular   if
  \(             k             <            n-r_i+1\),             then
  \(\mathcal{L}_i\cap \mathcal{B}_k=\emptyset \).
\end{proposition}

\begin{proof}
  Assume         by  contradiction         that
  \(x^\Lambda\partial_{k}\in            \mathcal{L}_i\)            and
  \(r_i \le \WD(x^\Lambda\partial_{k})\).  We
  have    that     \(i\)    is    minimum        such    that
  \(i=\lev_i(x^\Lambda               \partial_k)\).               Thus
  \(j  :=i-\WD(x^\Lambda\partial_k)  >  i   -(n-1)\)  is  such  that
  \[(h_i-2)(n-1) <  j=(h_i-1)(n-1)+r_i -\WD(x^\Lambda\partial_k) \le
  (h_i-1)(n-1).\]       Hence      \(h_j=(h_i-1)\)       and,      by
  Remark~\ref{rem:increase},
  \[j=i-\WD(x^\Lambda\partial_k)=\lev_i(x^\Lambda\partial_k)         -
  \WD(x^\Lambda\partial_k)=   \lev_j(x^\Lambda\partial_k),\]   where
  \(j<i\), a contradiction.
	
  If             \(       k      <       n-r_i+1\),       then
  \(r_i\le n-k  \le \WD(x^\Lambda\partial_k)\). From above, we have
  \(\mathcal{L}_i\cap \mathcal{B}_k=\emptyset \).
\end{proof}

The following corollary is straightforward.

\begin{corollary}\label{cor:deg_equiv_1_mod_n-1}
  If                                                    \(r_t=1\)~then
  \(\mathcal{L}_t             =
  \Set{x_1^{t+1}\partial_{n}}\).
\end{corollary}
\subsection{Periodicity}
We can now define a concept of \emph{periodicity} for the sequence \(\Set{\mathcal{L}_i}_{i\ge 0}\).
The following definition is crucial.

\begin{definition}\label{def:period}
  The                      \emph{period                      function}
  \(\per\colon \mathcal{B} \to \mathcal{B}\) is defined by
  \begin{equation*}   
  \per(x^\Lambda\partial_k) :=
    x_1^{n-1-\WD(x^\Lambda\partial_{k})}x^\Lambda\partial_k.
  \end{equation*} 
\end{definition}

\begin{remark}
  Note that if \(x^\Lambda\partial_k\in \mathcal{L}_i\), then
  \begin{align*}
    \lev_{i+n-1}(\per(x^\Lambda\partial_k))
    &=\lev_{i}(\per(x^\Lambda\partial_k))+ \WD(\per(x^\Lambda\partial_k)) \\
    &=\lev_{i}(\per(x^\Lambda\partial_k))+ \WD(x^\Lambda\partial_k) \\
    &=\lev_{i}(x^\Lambda\partial_k)+ \bigl(n-1 -\WD(x^\Lambda\partial_k)\bigr)+ \WD(x^\Lambda\partial_k) \\
    &=\lev_i(x^\Lambda\partial_{k}) +n-1=i+n-1,
  \end{align*}
  so   \(\per(x^\Lambda\partial_k)  \in \mathcal{L}_{i+n-1}\),  as
  \(t=i+n-1\) is  trivially  minimum   such that \(\lev_t(\per(x^\Lambda\partial_k))=t\). Moreover, since  the period map is  injective,
  we                             have                             that
  \(\Size{\mathcal{L}_{i}} \le \Size{\mathcal{L}_{i+n-1}} \).
\end{remark}
	
\begin{remark}
  Let  \(i\ge   1\) and   \(k>   u\ge  1\). Suppose   that
  \(x^\Lambda\partial_{k}\in       \mathcal{L}_i\)      and       
  \(x^\Theta\partial_{u}\in                  \mathcal{L}_{h_i(n-1)}\).
  Eq.~\eqref{eq:horizontal_operation}          shows         that
  \(\lambda_u               x^\Gamma\partial_k :=[x^\Lambda\partial_{k},
  x^\Theta\partial_{u}]\in\Z\mathcal{L}_i\),                      as
  \(\lev_i(x^\Gamma\partial_k)=\lev_i(x^\Lambda\partial_{k})\),
  whereas      Eq.~\eqref{eq:wd_decreasing}       shows      that
  \(\WD(x^\Gamma\partial_k) < \WD(x^\Lambda\partial_{k})\).  Note also
  that                           the                           element
  \(b_u=x_1^{h_i}x_{u-1}\partial_{u}  \in\mathcal{L}_{h_i(n-1)}\). Hence, given \(x^\Lambda\partial_{k}\in \mathcal{L}_i\), there exists a
  sequence \(u_1,u_2,\dots,u_s\) such that
  \[[x^\Lambda\partial_{k},  b_{u_1},b_{u_2},\dots,b_{u_s}] =  \lambda
    x_1^{i-h_i(n-k)+1}\partial_k\in \Z\mathcal{L}_i\] for a suitable
  non-zero   \(\lambda\in\Z\).   Eq.~\eqref{eq:wd_decreasing}
  shows                                                           that
  \(\WD([x^\Lambda\partial_{k},
  b_{u}])=\WD(x^\Lambda\partial_{k})-1\).                        Hence, from 
  Lemma~\ref{lem:bounds}, we have     \(s\le   k-2\).    Moreover,   if
  \(\lambda_{u} \ne 0\), then
  \begin{equation*}
    [x^\Lambda\partial_{k}, \underbrace{b_{u},\dots ,b_{u}}_{\text{\(\lambda_{u}\) times}},\underbrace{b_{u-1}, \dots,b_{u-1}}_{\text{\(\lambda_{u}\) times}},\dots,\underbrace{b_{2},\dots ,b_{2}}_{\text{\(\lambda_{u}\) times}}] 
    =                                                          \lambda
    x_u^{-\lambda_{u}}x_1^{\lambda_{u}(1+(h_i-1)(u-1))}x^\Lambda
    \partial_k
  \end{equation*}
  for     some     non-zero     \(\lambda\in     \Z\),     so     that
  \(\sum \lambda_{u}(u-1)  =s \le  k-2\), a condition which  is trivially
  equivalent  to \(\WD(x^\Lambda\partial_{k})  \le  n-2\), as  already
  seen in Lemma~\ref{lem:bounds}.
\end{remark}
The  previous remark  together with  Eq.~\eqref{eq:wd_decreasing}
and Lemma~\ref{lem:bounds} yield the following result.
\begin{proposition}
  If  \(\mathcal{L}_i  \cap  \mathcal{B}_k\ne \emptyset\),  then  there
  exists at  most one element in  \(\mathcal{L}_i \cap \mathcal{B}_k\)
  of  the form  \(x_1^t\partial_k\), in  which case  it is  the unique
  element        having         minimum        weight-degree        in
  \(\mathcal{L}_i     \cap      \mathcal{B}_k\).      The     exponent
  \(t=i-h_i(n-k)+1\)     is     determined    by     the     condition
  \(x_1^t\partial_k\in \mathcal{L}_{i}\).
\end{proposition}
	

Now we can prove one of the main contributions of this work where we give a precise characterization of the elements of \(\mathcal{L}_i\).
\begin{theorem}\label{thm:elements_of_Li_cap_Bk}
  Let \(i\ge 0\). A    basis    element     \(x^\Lambda\partial_{k}\)    belongs    to
  \(\mathcal{L}_{i}\cap             \mathcal{B}_k\)              if and only  if the
  following conditions are satisfied:
  \begin{enumerate}[(a)]
  \item \(n-k \le \WD(x^\Lambda\partial_{k}) < r_i\),
  \item \(i=\lev_i(x^\Lambda\partial_k) \).
  \end{enumerate}
\end{theorem}
	
\begin{proof}
  Suppose  first that (a)  and  (b)  are
  satisfied. We  have
  \(\lev_i(x^\Lambda\partial_{k})       =      i\),       and so
  \(x^\Lambda\partial_{k}\in  \mathcal{L}_j\cap  \mathcal{B}_k\),  for
  some  \(j\le   i\).   By contradiction,   assume  that  
  \(j=\lev_j(x^\Lambda\partial_k)\) and that \(i\ne j\). We have
  \(i>      j=(h_j-1)(n-1)+r_j  = \lev_j(x^\Lambda\partial_k)        =
  h_j\WD(x^\Lambda\partial_{k})  +   \deg(x^\Lambda)-1\).   Note  that at least one of the two conditions 
  \(h_i >h_j\) or \(r_i>r_j\) must be satisfied, so we have
  \begin{align*}
    r_i&= i-(h_i-1)(n-1)\\
       &= \lev_i(x^\Lambda\partial_k)-(h_i-1)(n-1)\\
       &= h_i\WD(x^\Lambda\partial_{k}) + \deg(x^\Lambda)-1 -(h_i-1)(n-1)\\
       &= (h_i-h_j)(\WD(x^\Lambda\partial_{k})-n+1) +\lev_j(x^\Lambda\partial_{k})  - (h_j-1)(n-1)\\
       &= (h_i-h_j)(\WD(x^\Lambda\partial_{k})-n+1) +j  - (h_j-1)(n-1)\\
       &= (h_i-h_j)(\WD(x^\Lambda\partial_{k})-n+1) +r_j\\
       & \le \begin{cases}
         \WD(x^\Lambda\partial_{k}) +(r_j-n+1)\le \WD(x^\Lambda\partial_{k}) < r_i & \text{if \(h_i > h_j\)}\\
         r_j < r_i & \text{if \(h_i = h_j\)},
       \end{cases}
  \end{align*}
  which is, in both the cases, a contradiction. Hence \(j=i \).
	
%
%
	
  Conversely,                       suppose                       that
  \(x^\Lambda\partial_{k}\in  \mathcal{L}_i\cap  \mathcal{B}_k\), so \(i= \lev_i(x^\Lambda\partial_{k})\).
  By Proposition~\ref{prop:rem_bound}  and Lemma~\ref{lem:bounds}
   follows \( n-k \le \WD(x^\Lambda\partial_{k}) < r_i\).
\end{proof}
\subsection{Connection to the sequence of partitions}
Let us give a description of the behavior of the chain \(\Set{\mathcal{N}_i}_{i}\) in terms of integer partitions.
Let  \(\{a_n\}_{n=0}^\infty\) be  the sequence  whose term  \(a_n\) is
equal   to   the   number   of  partitions   of   \(n\).    Let  also
\(b_n=\sum_{i=0}^na_i\) be the partial sum of $\{a_i\}$ and
\(c_n=\sum_{i=0}^nb_n\) be the partial sum of $\{b_i\}$, or the \emph{second partial sum} of $\{a_i\}$. The first values of the sequences are displayed in Table~\ref{tab:zero}, 
which also contains the corresponding OEIS references~\cite{OEIS}.

\begin{table}[hb]
		\centering
	{\renewcommand{\arraystretch}{1.3}
		\begin{tabular}{c||c|c|c|c|c|c|c|c|c|c|c|c|c|c|c||c}
			$i$ &0& 1 & $2$ & $3$ & $4$ & $5$ & $6$ & $7$ & $8$ & $9$ & 10& $11$ &$12$&13&14&OEIS\\
			\hline\hline
			${a_i}$&1& 1& 2& 3& 5& 7& 11& 15& 22& 30& 42& 56& 77& 101& 135&  \href{https://oeis.org/A000041}{A000041}\\ 
			${b_i}$&1&1& 3& 6& 11& 18& 29& 44& 66& 96& 138& 194& 271& 372& 507&  \href{https://oeis.org/A026905}{A026905}\\
			${c_i}$&1&1& 4& 10& 21& 39& 68& 112& 178& 274& 412& 606& 877& 1249& 1756&  \href{https://oeis.org/A085360}{A085360}\\ 
			\hline		
			\end{tabular}
		\bigskip }
	\caption{First values of the sequences $\{a_i\}, \{b_i\}$ and $\{c_i\}$}
	\label{tab:zero}
\end{table}

We are now ready to show that, after a threshold value with depends quadratically on $n$, the sequence $\{ \Size{\mathcal{L}_{i}}\}$ is periodic and how
it depends on $\{c_i\}$. Here the value $r_i$ is as in Definition~\ref{def:indices}.
\begin{corollary}\label{cor:lev_sizes}
  If \(i > (n-4)(n-1)\) and $1 \leq k \leq n$, then
  \begin{align*}
    \Size{\mathcal{L}_{i}\cap \mathcal{B}_k}&=b_{r_i+k-n-1}\\
    \intertext{and}
    \Size{\mathcal{L}_{i}}&= c_{r_i-1}.
  \end{align*}
\end{corollary}
\begin{proof}
 Let
  \(x^\Lambda\partial_{k}\in  \mathcal{L}_i\cap   \mathcal{B}_k\).  By
  Theorem~\ref{thm:elements_of_Li_cap_Bk} $x^\Lambda\partial_{k}$ satisfies
    \begin{itemize}
  \item[(a)]\(n-k\le \WD(x^\Lambda\partial_{k}) < r_i\),
  \item[(b)] \(i=\lev_i(x^\Lambda\partial_{k})\).
  \end{itemize}
  Letting \(\theta_u=\lambda_{u+1}\),  the condition (a) can  be rewritten
  as \(0\le \sum_{u\ge 1} u\theta_u < r_i+k-n\), where \(\theta_u\) is
  a  non-negative   integer.   There  are   exactly  \(b_{r_i+k-n-1}\)
  sequences   \(\Theta=\{\theta_u\}_{u=1}^\infty\),    including   the
  trivial one, satisfying this condition and determining the values of
  \(\lambda_u\) for \(u\ge 2\). The value of \(\lambda_1\) is uniquely
  determined by the condition  (b).  Indeed \(\WD(x^\Lambda\partial_{k})\)
  does  not depend  on  \(\lambda_1\)  and by   (b) we  have
  \(i=h_i\WD(x^\Lambda\partial_{k})+\lambda_1+\sum_{u\ge
    2}\lambda_u-1\).  Moreover,  by the hypotheses we  have \(h_i\ge n-3\).
  Thus
  \begin{align*}
    \lambda_1&=i - h_i\WD(x^\Lambda\partial_{k}) - \sum_{u\ge 2}\lambda_u+1 \\
             &=(h_i-1)(n-1)+r_i-h_i\WD(x^\Lambda\partial_{k}) - \sum_{u\ge 2}\lambda_u+1\\
             &=h_i\left(n-1-\WD(x^\Lambda\partial_{k})\right) -n+1 +r_i - \sum_{u\ge 2}\lambda_u+1 \\
             & \ge h_i\left(n-1-r_i+1\right) -n+1 +r_i - r_i-k+n+2\\
             & = h_i\left(n-r_i\right)  +3 -k \ge  h_i  +3 -k \ge n-k \ge 0
  \end{align*}	
  is uniquely determined and non-negative.
  
  Finally the equality  \(\Size{\mathcal{L}_{i}}= c_{r_i-1}\) is obtained
  computing
  \(\sum_{k=2}^{n} \Size{\mathcal{L}_{i}\cap \mathcal{B}_k}\).
\end{proof}
A straightforward consequence is the following result, which proves that the  sequence  \(\{\Size{\mathcal{L}_i}\}_{i\ge 0}\)  is  definitely
  periodic.
\begin{corollary}\label{cor:main}
  If  \(i >  (n-4)(n-1)\), then  the period  function $\per$ is a  bijection
 from  \(\mathcal{L}_i\) to  \(\mathcal{L}_{i+n-1}\).  
\end{corollary}

We conclude this section with an explicit example where we highlight the periodic structure of the sequence \(\Set{\mathcal{L}_i}\). 

\begin{example}
	Let \(n=5\) and  \(i\ge 5\).   We list  the
	non-empty sets \(\mathcal{L}_i\cap \mathcal{B}_k \).
	\begin{description}
		\item [\(r_i=1\)]
		\[\mathcal{L}_i\cap \mathcal{B}_5 = \Set{x_1^{i+1}\partial_5};\]
		\item [\(r_i=2\)]  \[
		\begin{split}
			\mathcal{L}_i\cap \mathcal{B}_5 &= \Set{x_1^{i+1}\partial_5,\, x_1^{i-h_i}x_2\partial_5}\\
			\mathcal{L}_i\cap              \mathcal{B}_{4}              &=
			\Set{x_1^{i+1-h_i}\partial_{4}};
		\end{split}
		\]
		\item [\(r_i=3\)]  \[
		\begin{split}
			\mathcal{L}_i\cap \mathcal{B}_5 &= \Set{x_1^{i+1}\partial_5,\, x_1^{i-h_i}x_2\partial_5,\, x_1^{i+1-2h_i}x_2^2\partial_5,\, x_1^{i-2h_i}x_3\partial_5}\\
			\mathcal{L}_i\cap \mathcal{B}_4 &= \Set{x_1^{i+1-h_i}\partial_4,\, x_1^{i-2h_i}x_2\partial_4}\\
			\mathcal{L}_i\cap              \mathcal{B}_{3}              &=
			\Set{x_1^{i+1-2h_i}\partial_{3}};
		\end{split}
		\]
		\item [\(r_i=4\)]  \[
		\begin{split}
			\mathcal{L}_i\cap \mathcal{B}_5 &= \Set{x_1^{i+1}\partial_5,\, x_1^{i-h_i}x_2\partial_5,\, x_1^{i+1-2h_i}x_2^2\partial_5,\, x_1^{i-2h_i}x_3\partial_5,\,\right.\\
				& \ \ \ \ \ \ \ \ \ \ \ \ \ \ \,\left. x_1^{i-2-3h_i}x_2^3\partial_5,\, x_1^{i-1-3h_i}x_2x_3\partial_5,\, x_1^{i-3h_i}x_4\partial_5}\\
			\mathcal{L}_i\cap \mathcal{B}_4 &= \Set{x_1^{i+1-h_1}\partial_4,\, x_1^{i-2h_i}x_2\partial_4,\, x_1^{i+1-3h_i}x_2^2\partial_4,\, x_1^{i-3h_i}x_3\partial_4}\\
			\mathcal{L}_i\cap \mathcal{B}_3 &= \Set{x_1^{i+1-2h_i}\partial_3,\, x_1^{i+1-3h_i}x_2\partial_3}\\
			\mathcal{L}_i\cap              \mathcal{B}_{2}              &=
			\Set{x_1^{i+1-3h_i}\partial_{2}}.
		\end{split}
		\]
	\end{description}
\end{example}

\section{The idealizer chain over the ring of integers}\label{sec:norm}
We conclude the paper by proving that the idealizer chain is generated by the subsets $\mathcal{N}_i$ of Definition~\ref{def:setchain}.
We start by describing the commutator structure of the chain \(\Set{\mathcal{N}_i}_{i\ge -1}\).


%

\begin{lemma}
	If              \(i              <             j\),              then
	\([\mathcal{N}_i,\mathcal{N}_j]             \subseteq             \Z
	\mathcal{N}_{j-1}\cup\Set{0}\).
\end{lemma}
\begin{proof}
	Let        \(x^\Lambda\partial_k\in       \mathcal{N}_i\)        and
	\(x^\Theta\partial_u\in     \mathcal{N}_j\)     be     such     that
	\([x^\Lambda\partial_{k},  x^\Theta\partial_{u}]  \ne  0\).   Assume
	first     that     \(x^\Theta\partial_u\ne     \partial_1\).      By
	Proposition~\ref{prop:lev_wd_equations} we have
	\begin{align}
		\lev_j([x^\Lambda\partial_{k},          x^\Theta\partial_{u}]) &=
		\lev_j(x^\Lambda\partial_{k})                                  +
		\lev_j(x^\Theta\partial_{u})-h_j(n-1) \nonumber \\
		&=\lev_i(x^\Lambda\partial_{k})                                  +
		\lev_j(x^\Theta\partial_{u})-h_j(n-1)+ \WD(x^\Theta\partial_{u})(h_j-h_i) \nonumber\\
		&< i+j -h_j(n-1) + (n-1)(h_j-h_i) \label{eq:ineq1}\nonumber\\
		&= i+j -h_j(n-1) = i-(n-1)+r_j \le i \le j-1.	\nonumber
	\end{align}
	From             this            we             have            that
	\([x^\Lambda\partial_{k}, x^\Theta\partial_{u}]  \in \mathcal{N}_j\)
	and     so     there     exists     \(h\le     j\)     such     that
	\(\lev_h([x^\Lambda\partial_{k}, x^\Theta\partial_{u}])=h\)  and, by
	the       above      argument,      \(h<j\).       This       yields
	\([x^\Lambda\partial_{k},          x^\Theta\partial_{u}]         \in
	\mathcal{N}_{j-1}\).

	If \(x^\Theta\partial_u=\partial_1\), then
	\begin{align*}
		\lev_i([x^\Lambda\partial_{k},          \partial_{1}])&= h_i(\WD(\partial_1( x^\Lambda)\partial_k))+\deg(\partial_1 (x^\Lambda))-1\\
		&= h_i(\WD(x^\Lambda\partial_k))+\deg(x^\Lambda)-2 = \lev_i(x^\Lambda\partial_k)-1 \le i -1 \le j-1.
	\end{align*}
	By     the      same     argument      as     above      we     have
	\([x^\Lambda\partial_{k},          x^\Theta\partial_{u}]         \in
	\mathcal{N}_{j-1}\).
\end{proof}
As straightforward consequence is the following corollary.
\begin{corollary}
	\(\mathcal{N}_j\subseteq N_{\mathcal{B}}(\Z\mathcal{N}_{j-1})\).
\end{corollary}
We prove now the opposite inclusion.

\begin{proposition}\label{prop:basis_normalizers}
	\(\mathcal{N}_j= N_{\mathcal{B}}(\Z\mathcal{N}_{j-1})\).
\end{proposition}

\begin{proof} 
	By   the    previous   corollary   it   suffices    to   show   that
	\(\mathcal{N}_j\supseteq       N_{\mathcal{B}}(\mathcal{N}_{j-1})\).
	Looking     for     a      contradiction,     we     assume     that
	\(x^\Lambda\partial_k\in   N_{\mathcal{B}}(\mathcal{N}_{j-1})\)   is
	such that \(\lev_i(x^\Lambda\partial_k)> i\)  for all \(i\le j\). Set
	\[x^\Theta\partial_u=:x_1^{(h_{j}-1)(u-\ell)}x_{\ell}\partial_{u},\]
	where \(1 \le  \ell < u\), and \(s:=(h_{j}-1)(n-1)\le  j-1\). We have
	\(h_s=h_{j}-1\) and
	\[
	\lev_s(x^\Theta\partial_u)=(h_{j}-1)(\ell-1+n-u)+(h_{j}-1)(u-\ell)=(h_{j}-1)(n-1)=s.
	\]
	Hence           both            \(x^\Theta\partial_u\)           and
	\([x^\Lambda\partial_k,x^\Theta\partial_u]\)        belong        to
	\(\mathcal{N}_{j-1}\).   Assume that  either \(\lambda_u\ne  0\) for
	some        \(1<u<n\)        or        \(l=k\)        so        that
	\([x^\Lambda\partial_k,x^\Theta\partial_u]\ne    0\).    For    some
	\(i             \le             j-1\)            we             have
	\(\lev_i([x^\Lambda\partial_k,x^\Theta\partial_u])=  i\).  We  start
	with assuming that  \(i>s\) and that \(h_j=h_{j-1}\),  which in turn
	implies \(h_i=h_{j-1}=h_j\). We have
	\[
	\begin{split}
		i&=\lev_i([x^\Lambda\partial_k,x^\Theta\partial_u])\\
		&= \lev_i(x^\Lambda\partial_k)+\lev_i(x^\Theta\partial_u)- h_i(n-1)\\
		&=\lev_j(x^\Lambda\partial_k) + h_{i}(\ell-1+n-u)+	(h_{j}-1)(u-\ell) - h_i(n-1)\\
		&=\lev_j(x^\Lambda\partial_k) + h_{j}(\ell-1+n-u)+	(h_{j}-1)(u-\ell) - h_{j}(n-1)\\
		&= \lev_j(x^\Lambda\partial_k) + (\ell-u) > j + (\ell-u).
	\end{split}
	\]
	As a consequence we  have \(j-1\ge i > j +  (\ell-u)\). Since we may
	alternatively  choose  \(\ell\)  or   \(u\),  with  respect  to  our
	assumption on  \(x^\Theta\partial_u\), such  that \(u-\ell  =1\), we
	have a contradiction.
	
	Suppose  now  that  \(i\le  s\) and  that  \(h_j=h_{j-1}\),  and  so
	\(h_i+1\le h_{j-1}=h_j\). We have
	\[
	\begin{split}
		i&=\lev_i([x^\Lambda\partial_k,x^\Theta\partial_u])\\
		&= \lev_i(x^\Lambda\partial_k)+\lev_i(x^\Theta\partial_u)- h_i(n-1)\\
		&= \lev_i(x^\Lambda\partial_k)+\lev_s(x^\Theta\partial_u) +(h_i-h_s)\WD(x^\Theta\partial_u)- h_i(n-1)\\
		&= \lev_i(x^\Lambda\partial_k)+\lev_s(x^\Theta\partial_u) +(h_i-h_s)(\ell-1+n-u)- h_i(n-1)\\
		&= \lev_i(x^\Lambda\partial_k)+h_s(n-1) +(h_i-h_s)(\ell-1+n-u)- h_i(n-1)\\
		&= \lev_i(x^\Lambda\partial_k) +(h_i-h_s)(\ell-u) \ge \lev_i(x^\Lambda\partial_k), \\
	\end{split}
	\]
	a contradiction.
	
	Suppose  now  that \(j-1=s\)  so  that  \(h_{j-1}=h_s=h_j-1\). By  a
	repeated   use  of   Lemma~\ref{lem:lev_plus_n-1}   we  may   assume
	\(\lev_t([x^\Lambda\partial_k,x^\Theta\partial_u])\le  t\) for  some
	\(t\) such that \(s-(n-1)< t \le s=j-1\). Hence
	\begin{align*}
		t&\ge \lev_t([x^\Lambda\partial_k,x^\Theta\partial_u])\\
		&= \lev_{s}([x^\Lambda\partial_k,x^\Theta\partial_u])\\
		&= \lev_{s}(x^\Lambda\partial_k)+\lev_{s}(x^\Theta\partial_u)- h_{s}(n-1)\\
		&= \lev_{j-1}(x^\Lambda\partial_k) + s -s = \lev_{j-1}(x^\Lambda\partial_k)
	\end{align*}
	giving                       the                       contradiction
	\(\lev_{j-1}(x^\Lambda\partial_k) \le t \le s= j-1\).
	
	We        are        now        left       with        the        case
	\(x^\Lambda\partial_k=   x_1^{\lambda_1}\partial_n\).     In           order           to           have
	\([x^\Lambda\partial_k,x^\Theta\partial_u]\ne   0\),    
	$x^\Theta\partial_u$ must be set to $\partial_1$.            In          particular
	\([x^\Lambda\partial_k,x^\Theta\partial_u]=                \lambda_1
	x_1^{\lambda_1-1}\partial_n\in  \Z   \mathcal{N}_{j-1}\),  and  also
	\(\lev_i(x_1^{\lambda_1-1}\partial_n)=\lambda_1-2 \)  is independent
	of    \(i\).     Thus    for    some   \(i\le    j-1\)    we    have
	\[\lev_{j-1}(x_1^{\lambda_1-1}\partial_n)=\lev_i(x_1^{\lambda_1-1}\partial_n)=\lambda_1-2
	\le            i             \le            j-1.\]             Hence
	\(j      \ge     \lambda_1-1=\lev_j(x^\Lambda\partial_k)>j\),      a
	contradiction.
\end{proof}

We are now able to prove the claimed result on the idealizer chain.

\begin{theorem}\label{thm:main}
	Let \(\mathfrak{M}_i\) be the free \(\Z\)-module spanned by
	\(\mathcal{N}_i\),  for $i \geq -1$.
	Then  \(\mathfrak{M}_i\) is  a
	homogeneous                      subring                      and
	\[\mathfrak{M}_i=N_{\mathfrak{L}(n)}(\mathfrak{M}_{i-1}) = \mathfrak{N}_i.\]
\end{theorem}
\begin{proof}
	The         statement          follows         directly         from
	Theorem~\ref{thm_homogeneous:normalizers}                        and
	Proposition~\ref{prop:basis_normalizers}         noting         that
	\(\mathfrak{T}=\Z\partial_1+\dots   +   \Z\partial_n =\mathfrak{N}_{-1}\)  is   an
	abelian homogeneous subring.
\end{proof}

\begin{remark}
	A   straightforward   consequence   of   Lemma~\ref{lem:bounds}   is   that
	\(\Lie{(n)} \ne  \bigcup_{i\ge -1}\mathfrak{N}_i\) for  \(n\ge 3\).
	For example, the element \(x_2^{3}\partial_3\), which  has weight-degree \(n\), cannot
	belong to any of the \(\mathfrak{N}_{i}\)s. We point out that, unlike the case of the Lie ring \(\Lie_{m}(n)\) (\(m>0\)), this shows that \(\Lie(n)\) is not nilpotent, beside not being  finitely generated. 
\end{remark}

We can now conclude the paper with the characterization of the idealizers of Eq.~\eqref{def_chain}.
\begin{theorem}{\label{thm:lev}}
	If                 \(i\ge                  0\),                 then
	\(x^\Lambda\partial_k \in \mathfrak{N}_i\setminus \mathfrak{N}_{i-1}
	\) if and only if \(i\)  is the least non-negative integer such that
	\( i = \lev_i(x^\Lambda\partial_k)
	\).  
\end{theorem}
\begin{proof}
The proof follows by noticing that,   by   Remark~\ref{rem:lev_equlity},   we    have   that
	\(x^\Lambda\partial_k \in \mathcal{N}_i\setminus \mathcal{N}_{i-1}\)
	if and  only if \(i\)  is the  least non-negative integer  such that
	\( i = \lev_i(x^\Lambda\partial_k) \).
\end{proof}

A trivial consequence of Corollary~\ref{cor:lev_sizes} is  the following conclusive result.
\begin{corollary}\label{coro:main2}
  For       \(i      >       (n-4)(n-1)\),     the       \(\Z\)-module
  \(\mathfrak{N}_i/\mathfrak{N}_{i-1}\) is free of rank \(c_{r_i-1}\).
\end{corollary}

\bibliographystyle{amsalpha} \bibliography{sym2n_ref.bib}
	
\end{document}

%% file: preamble.tex
\usepackage{amssymb,amsmath}
\usepackage[english]{babel}
\usepackage{amsfonts}
\usepackage{color}
\usepackage{amsthm}
\usepackage[colorlinks=true,linkcolor=NavyBlue, citecolor=NavyBlue,urlcolor=NavyBlue]{hyperref}
\usepackage{graphicx}
\usepackage{mathtools}
\usepackage[usenames,dvipsnames]{pstricks}
\usepackage{bm}
\usepackage{amsmath}
\usepackage{listings}
\usepackage{hyperref}
\usepackage[shortlabels]{enumitem}
\usepackage[draft]{fixme}

\usepackage{xy}
\usepackage[draft]{fixme}
\newcommand{\Size}[1]{\left\lvert #1 \right\rvert}
\newcommand{\Span}[1]{\left\langle#1\right\rangle}
\newcommand{\Set}[1]{\left\lbrace #1 \right\rbrace}

\let\phi\varphi

\theoremstyle{plain}
\newtheorem{theorem}{Theorem}[section]

\newtheorem{lemma}[theorem]{Lemma}
\newtheorem{proposition}[theorem]{Proposition}
\newtheorem{corollary}[theorem]{Corollary}

\theoremstyle{remark}
\newtheorem{remark}{Remark}

\theoremstyle{definition}
\newtheorem{definition}[theorem]{Definition}
\newtheorem{example}[theorem]{Example}
\newtheorem*{notation*}{Notation}

\usepackage{listings}
\lstdefinelanguage{GAP}{%
 morekeywords={%
 Assert,Info,IsBound,QUIT,%
 TryNextMethod,Unbind,and,break,%
 continue,do,elif,%
 else,end,false,fi,for,%
 function,if,in,local,%
 mod,not,od,or,%
 quit,rec,repeat,return,%
 then,true,until,while%
 },%
 sensitive,%
 morecomment=[l]\#,%
 morestring=[b]",%
 morestring=[b]',%
}[keywords,comments,strings]

\usepackage[T1]{fontenc}
\usepackage[variablett]{lmodern}
\usepackage{xcolor}
\DeclareMathOperator{\Z}{\mathbb{Z}}

\DeclareMathOperator{\Part}{\mathcal{P}}
\DeclareMathOperator{\Lie}{\mathfrak{L}}

\DeclareMathOperator{\Sym}{Sym}

\DeclareMathOperator{\wt}{wt}
\DeclareMathOperator{\WD}{wd}
\DeclareMathOperator{\lev}{lev}
\DeclareMathOperator{\per}{\nu}

\newenvironment{nouppercase}{%
  \renewcommand{\uppercasenonmath}[1]{}}{}
  
  \usepackage{collectbox}
\makeatletter